\documentclass[12pt]{article}
\usepackage{color}

\usepackage{amsmath,amsthm}
\usepackage{amsfonts}
\usepackage{latexsym}
\usepackage{indentfirst}
\usepackage{psfrag,epsf}
\usepackage{pstool}
\usepackage{tikz}
\usepackage{epsfig, subfigure}
\usepackage{graphicx}
\usepackage{epstopdf}
\usepackage{enumerate}
\usepackage[subnum]{cases}
\usepackage{geometry}
\usepackage{fullpage}
\usetikzlibrary{shapes,arrows,calc}

\newtheorem{theorem}{Theorem} \rm
\newtheorem{lemma}[theorem]{Lemma}

\newtheorem{conjecture}[theorem]{Conjecture}

\newtheorem{problem}[theorem]{Problem}

\newtheorem{question}[theorem]{Question}
\theoremstyle{plain}

\title{The square of every subcubic planar graph without 4-cycles and 5-cycles is 7-choosable}
\author{Ligang Jin\thanks{School of Mathematical Sciences, Zhejiang Normal University, Yingbin Road 688, 321004 Jinhua, China.  E-mail: {\tt ligang.jin@zjnu.cn}}
\and~Yingli Kang\thanks{Department of Mathematics, Jinhua University of Vocational Technology, Western Haitang Road 888, 321017 Jinhua, China.   E-mail: {\tt ylk8mandy@126.com}}
\and~Seog-Jin Kim\thanks{Department of Mathematics Education, Konkuk University,
Korea.  E-mail: {\tt skim12@konkuk.ac.kr}}~\thanks{ Corresponding author}
}

\begin{document}

\maketitle

\begin{abstract}
The square of a graph $G$, denoted by $G^2$, has the
same vertex set as $G$ and has an edge between two vertices if the distance between
them in $G$ is at most $2$.  Thomassen (2018)
and independently, Hartke, Jahanbekam and Thomas (2016)
proved that $\chi(G^2) \leq 7$ if $G$ is a subcubic planar graph.  A natural question is whether $\chi_{\ell}(G^2) \leq 7$ or not if $G$ is a subcubic planar graph.
Recently, Kim and Lian (2024) proved that $\chi_{\ell}(G^2) \leq 7$ if $G$ is a subcubic planar graph of girth  at least 6.
In this paper, we prove that $\chi_{\ell}(G^2) \leq 7$ if
$G$ is a subcubic planar graph without 4-cycles and 5-cycles, which improves the result of Kim and Lian.
\end{abstract}

\section{Introduction}

The {\em square} of a graph $G$, denoted by $G^2$, has the
same vertex set as $G$ and has an edge between two vertices if the distance between
them in $G$ is at most $2$.
We say that a graph $G$ is  {\em subcubic} if $\Delta(G) \leq 3$, where $\Delta(G)$ is the maximum degree in $G$.  The {\em girth} of a graph $G$, denoted by $g(G)$, is the size of the smallest cycle in $G$. A cycle of length $k$ is called a \emph{$k$-cycle}.
Let $\chi(G)$ denote the chromatic number of a graph $G$.

Given a graph $G$, finding the chromatic number of the square of $G$, $\chi(G^2)$, is a main research topic in the study of colorings of the square of graphs.
On coloring squares of graphs, one of the most famous problems is Wegner's conjecture, which states as follows.

\begin{conjecture}[\cite{Wegner}] \label{conj-Wegner}
Let $G$ be a planar graph. The chromatic number
$\chi(G^2)$ of $G^2$ is at most $7$ if $\Delta(G) = 3$,
at most $\Delta(G)+5$ if $4 \leq \Delta(G) \leq 7$, and at most $\lfloor \frac{3 \Delta(G)}{2} \rfloor$ if $\Delta(G) \geq 8$.
\end{conjecture}

Conjecture \ref{conj-Wegner} has received a lot of attention, but it is still widely open.  The only case for which we know tight bound is when $\Delta(G) = 3$. Thomassen \cite{Thomassen} showed that $\chi(G^2) \leq 7$ if $G$ is a planar graph with $\Delta(G)  = 3$, that is, Conjecture \ref{conj-Wegner} is true for $\Delta(G)  = 3$. Conjecture \ref{conj-Wegner} for $\Delta(G)  = 3$ was also confirmed by
Hartke, Jahanbekam and Thomas \cite{Hartke} with extensive computer case-checking.

With conditions on $\Delta(G)$, several upper bounds of $\chi(G^2)$ were obtained.
Bousquet, Deschamps, Meyer and Pierron \cite{BDMP} showed that $\chi(G^2) \leq 12$ if $G$ is a planar graph with $\Delta(G) \leq 4$.
For $\Delta(G) \leq 5$, recently the upper bounds on $\chi(G^2)$ were improved in \cite{Aoki, Deniz, Zou}, and now the best known upper bound was obtained in \cite{Deniz}.
Deniz \cite{Deniz} showed that $\chi(G^2) \leq 16$ if $G$ is a planar graph with $\Delta(G) \leq 5$.  Moreover, Bousquet, Deschamps, Meyer and Pierron \cite{BDMP21} showed that $\chi(G^2) \leq 2 \Delta(G) + 7$ if $G$ is a planar graph with
$6 \leq \Delta(G) \leq 31$.
 The best known upper bound for general $\Delta(G)$ is that $\chi(G^2) \leq \lceil \frac{5 \Delta(G)}{3}\rceil + 78$ by Molloy and Salavatipour \cite{MS05}.  On the other hand, Havet, van den Heuvel, McDiarmid, and Reed \cite{Havet} proved that
Conjecture \ref{conj-Wegner} holds asymptotically.
One may see  a detailed story on the study of Wegner's conjecture in \cite{Cranston22}.

A list assignment for a graph is a function $L$ that assigns to each vertex a list of
available colors. A graph is {\em $L$-colorable} if it has a proper coloring $f$ such that
$f (v) \in L(v)$ for all vertices $v$. A graph is called {\em $k$-choosable} if it is $L$-colorable whenever
all lists of $L$ have size $k$. The \emph{list chromatic number} $\chi_{\ell}(G)$ is the minimum $k$ such that $G$
is $k$-choosable.

Since it was known in \cite{Thomassen} that $\chi(G^2) \leq 7$ if $G$ is a subcubic planar graph, the following natural question was raised in \cite{CK} and \cite {Havet}, independently.

\begin{question}[\cite{CK,Havet}] \label{CK-question}
Is it true that  $\chi_{\ell} (G^2) \leq 7$ if $G$ is a subcubic planar graph?
\end{question}

Question \ref{CK-question} was motivated by the List Square Coloring Conjecture which states that
$\chi_{\ell}(G^2) = \chi(G^2)$ for every graph $G$.
However, the List Square Coloring Conjecture was disproved in \cite{KP}.
Note that a positive result for the List Square Coloring Conjecture for a special class of graphs is still interesting.
It was conjectured in \cite{Havet} that $\chi_{\ell}(G^2) = \chi(G^2)$ if $G$ is a planar graph.  But, recently Hasanvand \cite{MH} proved that there exists a cubic claw-free planar
graph $G$ such that $\chi(G^2) = 4 < \chi_{\ell}(G^2)$.

\medskip
So, it would be interesting to answer Question \ref{CK-question}.
But, considering the result of Thomassen in \cite{Thomassen}, it would be not easy to answer Question \ref{CK-question} completely if it is true. One natural approach to Question \ref{CK-question} is to forbid cycles of certain lengths.
For general upper bound on $\chi_{\ell}(G^2)$ for a subcubic graph $G$,
Cranston and Kim \cite{CK} proved that $\chi_{\ell} (G^2) \leq 8$ if $G$ is a connected graph (not necessarily planar) with $\Delta(G) = 3$ and if $G$ is not the Petersen graph.  Cranston and Kim  \cite{CK} also proved that $\chi_{\ell} (G^2) \leq 7$ if $G$ is a  subcubic planar graph with $g(G) \geq 7$.  Recently, Kim and Lian \cite{KL24} made an interesting progress by showing the following result.

\begin{theorem}[\cite{KL24}] \label{thm-KL}
If $G$ is a subcubic planar graph with girth at least 6,  then $\chi_{\ell}(G^2) \leq 7$.
\end{theorem}

\medskip

In this paper, we consider subcubic planar graphs that have no  4-cycles and  5-cycles, and
we improve Theorem \ref{thm-KL} by showing the following main theorem.

\begin{theorem} \label{main-thm}
If $G$ is a subcubic planar graph without 4-cycles and 5-cycles, then $\chi_{\ell}(G^2) \leq 7$.
\end{theorem}

\bigskip

In Section \ref{sec-proof}, we will prove several reducible configurations and prove Theorem \ref{main-thm}.


\section{Proof of Theorem \ref{main-thm}} \label{sec-proof}
In this section, let $G$ be a minimal counterexample to Theorem \ref{main-thm}.
It means that for any proper subgraph $H$ of $G$, $\chi_{\ell}(H^2) \leq 7$, but $\chi_{\ell}(G^2) > 7$.
Let $L$ be a list assignment of $G$ of size 7 such that $G^2$ is not $L$-colorable.  For a vertex $x$, $d_G(x)$ is the \emph{degree} of $x$ in $G$.  A vertex of degree $k$ is called a \emph{$k$-vertex}.
For vertices $x$ and $y$ in $G$, $d_G(x, y)$ is the \emph{distance} between $x$ and $y$ in $G$.

First, we state some basic properties of $G$, which were stated in previous papers and can be proved similarly as in \cite{KL24}.

\begin{lemma} \label{basic-lemma}
$G$ satisfies the following properties.
\begin{enumerate}[(a)]
\item $G$ has no 1-vertex. (Lemma 8 in \cite{KL24})

\item The distance between any two 2-vertices in $G$ is at least 4. (Lemma 10 in \cite{KL24})

\item If $u$ is a 2-vertex in $G$, then $u$ is not a cut-vertex in $G$. (Lemma 11 in \cite{KL24})

\end{enumerate}
\end{lemma}

Next, we state an important lemma which was proved in \cite{KL24}.

\begin{lemma} \label{lemma-KL} (Lemma 5 in \cite{KL24})
$G$ has no 6-cycle which contains a 2-vertex.
\end{lemma}


\medskip

Now, we prove some important new lemmas.

\begin{lemma} \label{C3-V2}
$G$ has no 3-cycle that contains a 2-vertex.
\end{lemma}
\begin{proof}
Suppose that $G$ has a 3-cycle $uvwu$ with $d_G(w)=2$.
Let $H = G - w$.  By the minimality of $G$, the square of $H$ has a proper
$L$-coloring $\phi$.  Note that $|\{x : xw \in E(G^2)\}| \leq 4$, since $G$ is a subcubic graph.
So, we can color $w$ by a color $c \in L(w) \setminus \{\phi(x) : xw \in E(G^2)\}$.
This gives an $L$-coloring of $G^2$, a contradiction.
\end{proof}

\begin{lemma} \label{C3-C6}
$G$ has no 6-cycle which is adjacent to a 3-cycle.
\end{lemma}
\begin{proof}
Suppose that $G$ has a 6-cycle which is adjacent to a 3-cycle.
Suppose that $G$ has
$H$ as a subgraph, and denote $V(H) = \{v_1, v_2, v_3, v_4, v_5, v_6, v_7\}$ (see Figure \ref{subgraph-H}).  Note that $\{v_1, v_2, v_3, v_4, v_5, v_6\}$ form a 6-cycle and $\{v_1,  v_6, v_7\}$
form a 3-cycle.

\begin{figure}[htbp]
\begin{center}
  \begin{tikzpicture}[u/.style={fill=black,minimum size =4pt,ellipse,inner sep=1pt},node distance=1.5cm,scale=0.8]
\node[u] (v1) at (150:2.1){};
\node[u] (v2) at (90:2.1){};
\node[u] (v3) at (30:2.1){};
\node[u] (v4) at (330:2.1){};
\node[u] (v5) at (270:2.1){};
\node[u] (v6) at (210:2.1){};
  \node[u] (v7) at (180:3.0) {};

\draw (v1) -- (v2);
\draw (v2) -- (v3);
\draw (v3) -- (v4);
\draw (v4) -- (v5);
\draw (v5) -- (v6);
\draw (v6) -- (v1);
\draw (v6) -- (v7);
\draw (v7) -- (v1);
\node[above=0.01cm] at (v1) {$v_1$};
   \node[above=0.01cm] at (v2) {$v_2$};
   \node[right=0.001] at (v3) {$v_3$};
  \node[below=0.01cm] at (v4) {$v_4$};
    \node[below=0.01cm] at (v5) {$v_5$};
  \node[below=0.01cm] at (v6) {$v_6$};
  \node[left=0.01cm] at (v7) {$v_7$};
\end{tikzpicture}
\caption{Subgraph $H$ } \label{subgraph-H}
\end{center}
\end{figure}
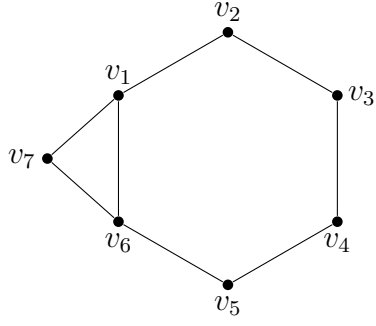

Let $L$ be a list assignment with lists of size 7 for each vertex in $G$.
We will show that $G^2$ has a proper coloring from the list $L$, which is a contradiction for the choice of $G$.

Let $H' = G - \{v_1, v_2, v_3, v_4, v_5, v_6, v_7\}$.  Then since $G$ is a minimal counterexample to Theorem \ref{main-thm}, the square of $H'$ has a proper coloring $\phi$ such that $\phi(v) \in L(v)$ for each vertex $v \in V(H')$.

For each $v_i \in V(H)$, we define \[
L_{H}(v_i) = L(v_i) \setminus \{\phi(x) : xv_i \in E(G^2) \mbox{ and } x \notin V(H)\}.
\]
Note that for any $x,y \in H$, if $d_H(x,y)\geq 3$, then $d_G(x,y)\geq 3$
since $G$ has neither 4-cycles nor 5-cycles.
Then, we have the following
\[|L_{H}(v_1)| \geq 5, |L_{H}(v_2)| \geq 3, |L_{H}(v_3)| \geq 2, |L_{H}(v_4)| \geq 2, |L_{H}(v_5)| \geq 3, |L_{H}(v_6)| \geq 5, |L_{H}(v_7)| \geq 4. \]

Next, we will show that
$H^2$ admits a proper coloring from the list $L_{H}$.

\medskip

\noindent {\bf Case 1}: $L_H (v_3) \cap L_H(v_7) \neq \emptyset$.

In this case, we color $v_3$ and $v_7$ by a color $c \in L_H (v_3) \cap L_H(v_7)$, and then greedily color $v_4, v_2, v_5, v_6, v_1$ in order.  Thus, $H^2$ has a proper coloring from the list.

\medskip

\noindent {\bf Case 2}: $L_H (v_3) \cap L_H(v_7) = \emptyset$.

In this case, we can color $v_2$ by a color $\alpha \in L_H(v_2)$ so that $|L_H(v_4) \setminus \{\alpha\}| \geq 2$.
Note that $\alpha \notin L_H(v_3)$ or $\alpha \notin L_H(v_7)$, since $L_H (v_3) \cap L_H(v_7) = \emptyset$.

\medskip
\noindent
{\bf Subcase 2.1}:  $\alpha \notin L_H(v_7)$.

Define \[
L'(v_i) = L_H(v_i) \setminus \{\alpha\} \mbox{ for } i = 1, 3, 4, 5, 6, 7.
\]
Then we have
\[|L'(v_1)| \geq 4, |L'(v_3)| \geq 1, |L'(v_4)| \geq 2, |L'(v_5)|  \geq 3, |L'(v_6)|  \geq 4,
|L'(v_7)|  \geq 4.\]
Now, greedily color $v_3, v_4, v_5, v_1, v_6, v_7$ in order. Thus, $H^2$ has a proper coloring from the list.

\medskip
\noindent
{\bf Subcase 2.2}:  $\alpha \in L_H(v_7)$ and $\alpha \notin L_H(v_3)$.

Define \[
L'(v_i) = L_H(v_i) \setminus \{\alpha\} \mbox{ for } i = 1, 3, 4, 5, 6, 7.
\]
Then we have
\[|L'(v_1)| \geq 4, |L'(v_3)| \geq 2, |L'(v_4)| \geq 2, |L'(v_5)|  \geq 3, |L'(v_6)|  \geq 4,
|L'(v_7)|  \geq 3.\]

Note that in this case, we may assume that $|L'(v_7)|  = 3$. Otherwise, $H^2$ has a proper coloring from the lists by the same argument as Subcase 2.1.  Then
there exists a color $\beta \in L'(v_1) \setminus L'(v_7)$ since $|L'(v_1)| >|L'(v_7)|$.

Now, color $v_3$ by a color $c_3 \in L'(v_3) \setminus \{\beta\}$, and then color $v_4$ and $v_5$
greedily.  Assume that $v_4$ is colored by $c_4$ and
$v_5$ is colored by $c_4$.  And define
\[
L''(v_1) = L'(v_1) \setminus \{c_3,c_5\},~L''(v_6) = L'(v_6) \setminus \{c_4,c_5\}, \text{~and~} L''(v_7) = L'(v_7) \setminus \{c_5\}.
\]
Note that $|L''(v_i)| \geq 2$ for $i = 1, 6, 7$.

If $c_5 = \beta$, then $|L''(v_7)| \geq 3$ since $c_5 = \beta \notin L'(v_7)$.  So, we can greedily color $v_1, v_6, v_7$ in order, which completes the coloring of $H^2$.

If $c_5 \neq \beta$, then $\beta\in L''(v_1)$. Thus, we can color $v_1$ with the color $\beta$ and then greedily color $v_6$ and $v_7$ in order, which completes the coloring of $H^2$.

By Case 1 and Case 2,
$H^2$ has a proper coloring from the list $L_H$.
Thus $G^2$ has a proper coloring from the list $L$, which is a contradiction. Hence, no 6-cycle is adjacent to a 3-cycle in $G$.  This completes the proof of Lemma \ref{C3-C6}.
\end{proof}

Next, we will show that the distance between any two 3-cycles is at least 3.

\begin{lemma} \label{lem-distance-cycle}
The distance between any two 3-cycles is at least 3 in $G$.
\end{lemma}
\begin{proof}

Suppose to the contrary that there are two 3-cycles whose distance is at most 2.
Note that any two 3-cycles are not adjacent in $G$, since $G$ has no 4-cycles.
Also any two 3-cycles are not intersecting in $G$, since $G$ is subcubic.

\medskip
\noindent {\bf Case 1}: Distance of two 3-cycles is 1.

Suppose that $G$ has $W_1$ in Figure \ref{subgraph-W} as a subgraph, and denote $V(W_1) = \{v_1, v_2, v_3, v_4, v_5, v_6\}$.
Note that the distance of the two 3-cycles $v_1 v_2 v_3 v_1$ and $v_4v_5v_6v_4$ is 1.

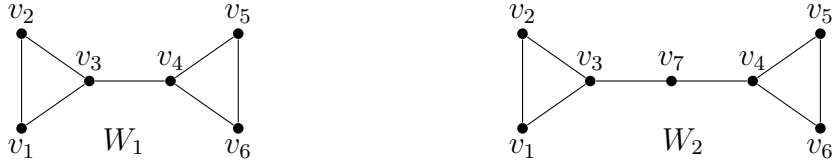
\begin{figure}[htbp]
\begin{center}
  \begin{tikzpicture}[u/.style={fill=black,minimum size =4pt,ellipse,inner sep=1pt},invisible/.style={circle,draw=none,fill=none,inner sep=0pt,minimum size=0pt},node distance=1.5cm,scale=0.9]

\node[u] (v1) at (0,0){};
\node[u] (v2) at (-1,0.7){};
\node[u] (v3) at (-1,-0.7){};
\node[u] (v5) at (1.2,0){};
\node[u] (v6) at (2.2,0.7) {};
\node[u] (v7) at (2.2,-0.7) {};

\draw  (v1) to (v2){};
\draw  (v1) to (v3){};
\draw  (v1) to (v5){};
\draw  (v2) to (v3){};
\draw  (v5) to (v6){};
\draw  (v5) to (v7){};
\draw  (v6) to (v7){};

\node[above] at (v1) {$v_3$};
\node[below] at (v3) {$v_1$};
\node[above] at (v2) {$v_2$};
\node[above] at (v5) {$v_4$};
\node[above] at (v6) {$v_5$};
\node[below] at (v7) {$v_6$};
\node[below=0.8cm,left=0.2cm] at (v5) {$W_1$};
\end{tikzpicture}\hspace{3cm}
  \begin{tikzpicture}[u/.style={fill=black,minimum size =4pt,ellipse,inner sep=1pt},invisible/.style={circle,draw=none,fill=none,inner sep=0pt,minimum size=0pt},node distance=1.5cm,scale=0.9]

\node[u] (v1) at (0,0){};
\node[u] (v2) at (-1,0.7){};
\node[u] (v3) at (-1,-0.7){};
\node[u] (v5) at (2.4,0){};
\node[u] (v6) at (3.4,0.7) {};
\node[u] (v7) at (3.4,-0.7) {};
\node[u] (v8) at (1.2,0) {};

\draw  (v1) to (v2){};
\draw  (v1) to (v3){};
\draw  (v1) to (v5){};
\draw  (v2) to (v3){};
\draw  (v5) to (v6){};
\draw  (v5) to (v7){};
\draw  (v6) to (v7){};

\node[above] at (v1) {$v_3$};
\node[below] at (v3) {$v_1$};
\node[above] at (v2) {$v_2$};
\node[above] at (v5) {$v_4$};
\node[above] at (v6) {$v_5$};
\node[below] at (v7) {$v_6$};
\node[above] at (v8) {$v_7$};
\node[below=0.8cm,left=0.5cm] at (v5) {$W_2$};
\end{tikzpicture}
\caption{Subgraph $W_1$ and $W_2$. } \label{subgraph-W}
\end{center}
\end{figure}

Let $W_1' = G - \{v_1, v_2, v_3, v_4, v_5, v_6\}$.  Then since $G$ is a minimal counterexample to Theorem \ref{main-thm}, the square of $W_1'$ has a proper $L$-coloring $\phi$ such that $\phi(v) \in L(v)$ for each vertex $v \in V(W_1')$.

Now, for each $v_i \in V(W_1)$, we define \[
L_{W_1}(v_i) = L(v_i) \setminus \{\phi(x) : xv_i \in E(G^2) \mbox{ and } x \notin V(W_1)\}.
\]
Note that for any $x,y \in W_1$, if $d_{W_1}(x,y)\geq 3$, then $d_G(x,y)\geq 3$
since $G$ has neither 4-cycles nor 5-cycles.
Then, we have the following
\[|L_{W_1}(v_1)| \geq 3, |L_{W_1}(v_2)| \geq 3, |L_{W_1}(v_3)| \geq 5, |L_{W_1}(v_4)| \geq 5,
|L_{W_1}(v_5)| \geq 3, |L_{W_1}(v_6)| \geq 3.\]

Next, we will show that
$W_1^2$ admits a proper coloring from the list $L_{W_1}$.
If $|L_{W_1}(v_6)| =3$, then color $v_3$ by a color $\alpha \in L_{W_1}(v_3)\setminus L_{W_1}(v_6)$; otherwise, choose $\alpha \in L_{W_1}(v_3)$ arbitrarily.
Then, we have that $|L_{W_1}(v_6) \setminus \{\alpha\}| \geq 3$.
Next, we can greedily color $v_1, v_2, v_4, v_5, v_6$ in order.
Thus $W_1^2$ admits a proper coloring from the list $L_{W_1}$. So,
$G^2$ has a proper coloring from the list $L$, which is a contradiction.

\bigskip
\noindent {\bf Case 2}: Distance of two 3-cycles is 2.

Suppose that $G$ has $W_2$ in Figure \ref{subgraph-W} as a subgraph, and denote $V(W_2) = \{v_1, v_2, v_3, v_4, v_5, v_6, v_7\}$.
Note that the distance of the two 3-cycles $v_1 v_2 v_3 v_1$ and $v_4v_5v_6v_4$ is 2.

Let $W_2' = G - \{v_1, v_2, v_3, v_4, v_5, v_6, v_7\}$.  Then since $G$ is a minimal counterexample to Theorem \ref{main-thm}, the square of $W_2'$ has a proper coloring $\phi$ such that $\phi(v) \in L(v)$ for each vertex $v \in V(W_2')$.

Now, for each $v_i \in V(W_2)$, we define \[
L_{W_2}(v_i) = L(v_i) \setminus \{\phi(x) : xv_i \in E(G^2) \mbox{ and } x \notin V(W_2)\}.
\]
Note that for any $x,y \in W_2$, if $d_{W_2}(x,y)\geq 3$, then $d_G(x,y)\geq 3$
since $G$ has neither 4-cycles nor 5-cycles.
Then, we have the following
\[|L_{W_2}(v_1)| \geq 3, |L_{W_2}(v_2)| \geq 3, |L_{W_2}(v_3)| \geq 4, |L_{W_2}(v_4)| \geq 4,
|L_{W_2}(v_5)| \geq 3, |L_{W_2}(v_6)| \geq 3, |L_{W_2}(v_7)| \geq 4.\]

Next, we will show that
$W_2^2$ admits a proper coloring from the list $L_{W_2}$.
If $|L_{W_2}(v_6)| =3$, then color $v_7$ by a color $\beta \in L_{W_2}(v_7)\setminus L_{W_2}(v_6)$; otherwise, choose $\beta \in L_{W_2}(v_7)$ arbitrarily.
Note that $|L_{W_2}(v_6) \setminus \{\beta\}| \geq 3$.
And then greedily color $v_1, v_2, v_3, v_4, v_5, v_6$ in order.
So, $G^2$ has a proper coloring from the list $L$, which is a contradiction.

Thus by Case 1 and Case 2, the distance between any two 3-cycles is at least 3.
\end{proof}

Next, we will show that the distance between a 3-cycle and a 2-vertex is at least 4.

\begin{lemma} \label{lem-distance-cycle-vertex}
The distance between a 3-cycle  and a 2-vertex is at least 4 in $G$.
\end{lemma}
\begin{proof}
Suppose to the contrary that $G$ has a 3-cycle $v_1v_2v_3v_1$ and a 2-vertex $w$ whose distance (denoted by $d$) is at most 3.
Note that $w$ is not on the 3-cycle $v_1v_2v_3v_1$ by Lemma \ref{C3-V2}.

\medskip
\noindent {\bf Case 1}: $d=1$.

Suppose that $G$ has $Q_1$ in Figure \ref{subgraph-Q} as a subgraph, and denote $V(Q_1) = \{v_1, v_2, v_3, w\}$.
Note that the distance of the 3-cycle $v_1 v_2 v_3 v_1$ and the 2-vertex $w$ is 1.

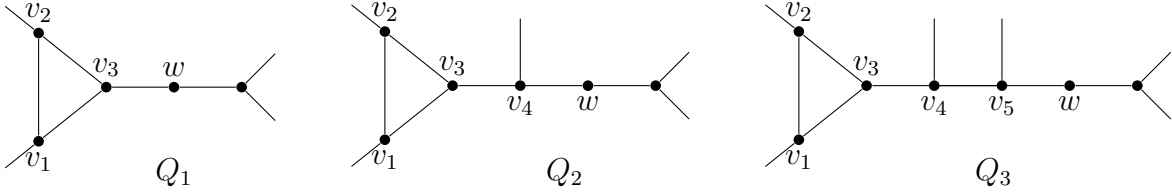
\begin{figure}[htbp]
\begin{center}
  \begin{tikzpicture}[u/.style={fill=black,minimum size =4pt,ellipse,inner sep=1pt},invisible/.style={circle,draw=none,fill=none,inner sep=0pt,minimum size=0pt},node distance=1.5cm,scale=0.9]

\node[u] (v1) at (0,0){};
\node[u] (v2) at (-1,0.8){};
\node[u] (v3) at (-1,-0.8){};
\node[u] (v4) at (1,0){};
\node[u] (v5) at (2,0){};
\node[invisible] (v6) at (2.5,0.5) {};  
\node[invisible] (v7) at (2.5,-0.5) {};  
\node[invisible] (v8) at (-1.5,1.2) {};  
\node[invisible] (v9) at (-1.5,-1.2) {};

\draw  (v1) to (v2){};
\draw  (v1) to (v3){};
\draw  (v1) to (v4){};
\draw  (v2) to (v3){};
\draw  (v2) to (v8){};
\draw  (v9) to (v3){};
\draw  (v5) to (v4){};
\draw  (v5) to (v6){};
\draw  (v5) to (v7){};

\node[above] at (v1) {$v_3$};
\node[below] at (v3) {$v_1$};
\node[above] at (v2) {$v_2$};
\node[above] at (v4) {$w$};
\node[below=0.8cm] at (v4) {$Q_1$};
\end{tikzpicture}\hspace{1cm}\begin{tikzpicture}[u/.style={fill=black,minimum size =4pt,ellipse,inner sep=1pt},invisible/.style={circle,draw=none,fill=none,inner sep=0pt,minimum size=0pt},node distance=1.5cm,scale=0.9]

\node[u] (v1) at (0,0){};
\node[u] (v2) at (-1,0.8){};
\node[u] (v3) at (-1,-0.8){};
\node[u] (v4) at (2,0){};
\node[u] (v5) at (3,0){};
\node[u] (v10) at (1,0){};
\node[invisible] (v6) at (3.5,0.5) {};  
\node[invisible] (v7) at (3.5,-0.5) {};  
\node[invisible] (v8) at (-1.5,1.2) {};   
\node[invisible] (v9) at (-1.5,-1.2) {};  
\node[invisible] (v11) at (1,1) {};

\draw  (v1) to (v2){};
\draw  (v1) to (v3){};
\draw  (v1) to (v4){};
\draw  (v2) to (v3){};
\draw  (v2) to (v8){};
\draw  (v9) to (v3){};
\draw  (v5) to (v4){};
\draw  (v5) to (v6){};
\draw  (v5) to (v7){};
\draw  (v10) to (v11){};

\node[above] at (v1) {$v_3$};
\node[below] at (v3) {$v_1$};
\node[above] at (v2) {$v_2$};
\node[below] at (v4) {$w$};
\node[below] at (v10) {$v_4$};
\node[below=1.15cm,right=0.2cm] at (v10) {$Q_2$};
\end{tikzpicture}\hspace{1cm}\begin{tikzpicture}[u/.style={fill=black,minimum size =4pt,ellipse,inner sep=1pt},invisible/.style={circle,draw=none,fill=none,inner sep=0pt,minimum size=0pt},node distance=1.5cm,scale=0.9]

\node[u] (v1) at (0,0){};
\node[u] (v2) at (-1,0.8){};
\node[u] (v3) at (-1,-0.8){};
\node[u] (v4) at (3,0){};
\node[u] (v5) at (4,0){};
\node[u] (v10) at (1,0){};
\node[u] (v12) at (2,0){};
\node[invisible] (v13) at (2,1){};
\node[invisible] (v6) at (4.5,0.5) {};
\node[invisible] (v7) at (4.5,-0.5) {};
\node[invisible] (v8) at (-1.5,1.2) {};
\node[invisible] (v9) at (-1.5,-1.2) {};
\node[invisible] (v11) at (1,1) {};

\draw  (v1) to (v2){};
\draw  (v1) to (v3){};
\draw  (v1) to (v4){};
\draw  (v2) to (v3){};
\draw  (v2) to (v8){};
\draw  (v9) to (v3){};
\draw  (v5) to (v4){};
\draw  (v5) to (v6){};
\draw  (v5) to (v7){};
\draw  (v10) to (v11){};
\draw  (v10) to (v12){};
\draw  (v12) to (v13){};
\node[above] at (v1) {$v_3$};
\node[below] at (v3) {$v_1$};
\node[above] at (v2) {$v_2$};
\node[below] at (v4) {$w$};
\node[below] at (v10) {$v_4$};
\node[below] at (v12) {$v_5$};
\node[below=1.15cm,right=0.4cm] at (v10) {$Q_3$};
\end{tikzpicture}
\caption{Subgraph $Q_1$, $Q_2$, and $Q_3$ } \label{subgraph-Q}
\end{center}

\end{figure}

Let $Q_1' = G - \{v_1, v_2, v_3, w\}$.  Then since $G$ is a minimal counterexample to Theorem \ref{main-thm}, the square of $Q_1'$ has a proper coloring $\phi$ such that $\phi(v) \in L(v)$ for each vertex $v \in V(Q_1')$.

Now, for each $v \in V(Q_1)$, we define \[
L_{Q_1}(v) = L(v) \setminus \{\phi(x) : xv \in E(G^2) \mbox{ and } x \notin V(Q_1)\}.
\]
Note that for any $x,y \in Q_1$, if $d_{Q_1}(x,y)\geq 3$, then $d_G(x,y)\geq 3$
since $G$ has neither 4-cycles nor 5-cycles.
Then, we have the following
\[|L_{Q_1}(v_1)| \geq 3, |L_{Q_1}(v_2)| \geq 3, |L_{Q_1}(v_3)| \geq 4, |L_{Q_1}(w)| \geq 4.\]

In this case, greedily color $v_1, v_2, v_3, w$ in order.
Then the square of $Q_1$ has a proper coloring from the list $L_{Q_1}$.
So $G^2$ has a proper coloring from the list $L$, which is a contradiction.

\medskip
\noindent {\bf Case 2}: $d=2$.

Suppose that $G$ has $Q_2$ in Figure \ref{subgraph-Q}  as a subgraph, and denote $V(Q_2) = \{v_1, v_2, v_3, v_4,  w\}$.
Note that the distance of the 3-cycle $v_1 v_2 v_3 v_1$ and the 2-vertex $w$ is 2.

Let $Q_2' = G - \{v_1, v_2, v_3, v_4, w\}$.  Then since $G$ is a minimal counterexample to Theorem \ref{main-thm}, the square of $Q_2'$ has a proper coloring $\phi$ such that $\phi(v) \in L(v)$ for each vertex $v \in V(Q_2')$.

Now, for each $v \in V(Q_2)$, we define \[
L_{Q_2}(v) = L(v) \setminus \{\phi(x) : xv \in E(G^2) \mbox{ and } x \notin V(Q_2)\}.
\]
Note that for any $x,y \in Q_2$, if $d_{Q_2}(x,y)\geq 3$, then $d_G(x,y)\geq 3$
since $G$ has neither 4-cycles nor 5-cycles.
Then, we have the following
\[|L_{Q_2}(v_1)| \geq 3, |L_{Q_2}(v_2)| \geq 3, |L_{Q_2}(v_3)| \geq 4, |L_{Q_2}(v_4)| \geq 3, |L_{Q_2}(w)| \geq 3.\]

In this case, color $v_3$ by a color $c \in L_{Q_2}(v_3)$ such that $|L_{Q_2}(v_1)\setminus \{c\}|\geq 3$.  And then
greedily color $w, v_4, v_2, v_1$ in order.
Then the square of $Q_2$ has a proper coloring from the list $L_{Q_2}$.
So $G^2$ has a proper coloring from the list $L$, which is a contradiction.

\medskip
\noindent {\bf Case 3}: $d=3$.

Suppose that $G$ has $Q_3$ in Figure \ref{subgraph-Q}  as a subgraph, and denote $V(Q_3) = \{v_1, v_2, v_3, v_4, v_5,  w\}$.
Note that the distance of the 3-cycle $v_1 v_2 v_3 v_1$ and the 2-vertex $w$ is 3.

Let $Q_3' = G - \{v_1, v_2, v_3, v_4, v_5, w\}$.  Then since $G$ is a minimal counterexample to Theorem \ref{main-thm}, the square of $Q_3'$ has a proper coloring $\phi$ such that $\phi(v) \in L(v)$ for each vertex $v \in V(Q_3')$.

Now, for each $v \in V(Q_3)$, we define \[
L_{Q_3}(v) = L(v) \setminus \{\phi(x) : xv \in E(G^2) \mbox{ and } x \notin V(Q_3)\}.
\]
Note that for any $x,y \in Q_3$, if $d_{Q_3}(x,y)\geq 3$, then $d_G(x,y)\geq 3$
since $G$ has neither 4-cycles nor 5-cycles.
Then, we have the following
\[|L_{Q_3}(v_1)| \geq 3, |L_{Q_3}(v_2)| \geq 3, |L_{Q_3}(v_3)| \geq 4, |L_{Q_3}(v_4)| \geq 3,
|L_{Q_3}(v_5)| \geq 2, |L_{Q_3}(w)| \geq 3.\]
We can color $v_3$ with a color $\alpha \in L_{Q_3}(v_3)$
so that $|L_{Q_3}(v_1) \setminus \{\alpha\}| \geq 3$.
Then we can greedily color $v_5$, $v_4$, $w$, $v_2$, $v_1$ in order.
Thus, the square of $Q_3$ has a proper coloring from the list $L_{Q_3}$.
So $G^2$ has a proper coloring from the list $L$, which is a contradiction.
\end{proof}

Now we state a key lemma which is important in discharging part.

\begin{lemma} \label{key-lem-discharging}
For each integer $\ell \geq 7$, let $C_{\ell}$ be a face of $G$ of length $\ell$.  Then the sum of the number of 3-cycles adjacent to $C_{\ell}$ and the number of 2-vertices on $C_{\ell}$ is at most ${\displaystyle \big\lfloor \frac{\ell}{4} \big\rfloor }$.
\end{lemma}

\begin{proof}
The result is derived from Lemma \ref{basic-lemma}(b), Lemma \ref{lem-distance-cycle} and Lemma \ref{lem-distance-cycle-vertex}.
\end{proof}


Now we prove the main theorem.

\medskip
\noindent {\bf Theorem \ref{main-thm}.}
If $G$ is a subcubic planar graph without 4-cycles and 5-cycles, then $\chi_{\ell}(G^2) \leq 7$.
\begin{proof}
Let $G$ be a minimal counterexample to the theorem and let $G$ be a plane graph drawn on the plane without crossing edges.  Let $F(G)$ be the set of faces of $G$. For a face $C \in F(G)$, let $\ell(C)$ be the length of $C$.

We assign $2d(x)-6$ to each vertex $x \in V(G)$ and $\ell(x) - 6$ for each face $x \in F(G)$ as an original charge function $\omega(x)$ of $x$.
According to Euler's formula $|V(G)| - |E(G)| + |F(G)| = 2$,
we have
\begin{equation*} \label{eqn1}
\sum_{x\in V(G)\cup F(G)}\omega(x) = \sum_{v\in V(G)}(2d(v)-6)+\sum_{f\in F(G)}(\ell(f)-6) = -12.
\end{equation*}
We next design some discharging rules to redistribute charges along the graph with conservation of the total charge. Let $\omega'(x)$ be the charge of $x \in V(G)\cup F(G)$ after the discharge procedure. Thus, ${\displaystyle \sum_{x\in V(G)\cup F(G)}\omega(x)=\sum_{x\in V(G)\cup F(G)}\omega'(x)}$. Next, we will show that $\omega'(x)\ge 0$ for all $x\in V(G)\cup F(G)$, which leads to the following contradiction
\[
0\le \sum_{x\in V(G)\cup F(G)}\omega'(x)=\sum_{x\in V(G)\cup F(G)}\omega(x) = \sum_{v\in V(G)}(2d(v)-6)+\sum_{f\in F(G)}(\ell(f)-6) = -12.
\]

\medskip
\noindent {\bf The discharging rules:}
\begin{enumerate}[(R1)]
\item If a 2-vertex $u$ is on a face $C$ with $\ell(C)\geq 7$, then $C$ gives charge 1 to $u$.

\item If a 3-face $C_3$ is adjacent to a face $C$ with $\ell(C)\geq 7$, then $C$ gives charge 1 to $C_3$.
\end{enumerate}
\medskip

Next, we will show that the new charge $\omega'(x) \geq 0$ for every $x \in V(G) \cup F(G)$.

\medskip
\noindent
{\bf Case 1:} Let $x \in V(G)$.

Observe that $G$ has no 1-vertex by Lemma~\ref{basic-lemma} (a).
If $d(x)  = 2$, then  $\omega(x)  = -2$.
Note that every 2-vertex is incident to two faces by Lemma \ref{basic-lemma} (c).
Also note that both incident faces of $x$ have length at least 7
by Lemma \ref{lemma-KL} and Lemma \ref{C3-V2}.
By (R1),  $x$ receives charge 1 from each of its incident faces.  So, $\omega'(x) = \omega(x)+1\times 2=0$.
If $d(x) = 3$, then $\omega'(x) = \omega(x) = 0$.

\medskip
\noindent {\bf Case 2:} Let $x \in F(G)$.

Here we denote $x$ by a face $C$. Recall that $G$ has neither a 4-cycle nor a 5-cycle.
If $\ell(C)= 3$, then $C$ is adjacent to three faces of length at least 7 by Lemma \ref{C3-V2} and Lemma \ref{C3-C6}. By (R2), $C$ receives charge 1 from each of its adjacent faces. So, $\omega'(C) = \omega(C)+1\times 3=0$.
If $\ell(C)= 6$, then $\omega'(C)  = \omega(C) = 0$.

Now, we consider the case when $\ell(C)\ge 7$.
By Lemma \ref{key-lem-discharging}, the sum of the number of 2-vertices on $C$ and the number of 3-faces adjacent to $C$ is at most $\lfloor \frac{\ell(C)}{4} \rfloor$.
So by (R1) and (R2), $\omega'(C) \geq \ell(C) - 6 - \lfloor \frac{\ell(C)}{4} \rfloor \geq 0$.

\medskip
Hence, by Cases 1 and 2, we have that $\omega'(x) \geq 0$ for every $x \in V(G) \cup F(G)$.  This completes the proof of Theorem \ref{main-thm}.
\end{proof}

\section{Future work}
We proved that $\chi_{\ell}(G^2) \leq 7$ if $G$ is a subcubic planar graph without 4-cycles and 5-cycles. But, it is unknown whether Question \ref{CK-question} is true or not.  So, it is interesting to answer Question \ref{CK-question} completely.
Or as a weaker version, we can ask the following problem.

\begin{problem}
Is it true that $\chi_{\ell}(G^2) \leq 7$ if $G$ is a subcubic planar graph of girth at least 4?
\end{problem}

On the other hand, we can consider an extension to DP-coloring version.
For detailed information on DP-coloring, one may refer to \cite{JKZ}.

\begin{problem}
Is it true that $\chi_{DP}(G^2) \leq 7$ if $G$ is a subcubic planar graph?  Or, is there a subcubic planar graph $G$ such that $\chi_{DP}(G^2) > 7$?
\end{problem}

\section*{Acknowledgments}
The third author thanks Xuding Zhu for his support and hospitality while he was visiting Zhejiang Normal University in 2025.
The first author is supported by National Natural Science Foundation of China (Grant No. 11801522).
This paper was written as part of Konkuk University's research support program for its faculty on sabbatical leave in 2025 (S.-J. Kim).



\end{document}